%% file: nonregular-art.tex
\documentclass[12pt]{article}
\usepackage{fullpage,amsthm,amsmath,amssymb,xspace,setspace,url,tikz}
\usetikzlibrary{positioning, shapes.geometric, calc}

\newcounter{thmctr}
\newtheorem{thm}[thmctr]{Theorem}
\newtheorem{lemma}[thmctr]{Lemma}
\newtheorem{prop}[thmctr]{Proposition}

\theoremstyle{definition}
\newtheorem*{definition}{Definition}

\newtheorem*{remarks}{Remarks}

\theoremstyle{plain}

\tikzstyle{pathscale}=[]
\tikzstyle{vertex}=[circle,fill=black,inner sep=2pt]
\tikzstyle{vertrect}=[draw,rectangle,inner sep=2pt]
\tikzstyle{vertdia}=[draw,diamond,inner sep=2pt]
\newcommand{\dotp}[2]{\left< #1, #2 \right>}

\newcommand{\propeig}[1]{\texttt{Eig}[#1]\xspace}
\newcommand{\propexpand}[1]{\texttt{Expand}[#1]\xspace}
\newcommand{\propcount}[1]{\texttt{Count}[#1]\xspace}
\newcommand{\propcycle}[2][4]{\texttt{Cycle$_{#1}$}[#2]\xspace}

\newcommand{\tr}[1]{\text{Tr}\left[#1\right]}
\newcommand{\prooftext}{Proof\xspace}

\newcommand{\powerscountcyclescite}{\cite[Proposition 6]{hqsi-lenz-quasi12}\xspace}

\newcommand{\dhruvuni}{University of Illinois at Chicago \\ mubayi@uic.edu}
\newcommand{\johnuni}{University of Illinois at Chicago \\ lenz@math.uic.edu}
\newcommand{\dhruvfoot}{\footnote{Research supported in part by  NSF Grants 0969092 and 1300138.}}
\newcommand{\johnfoot}{\footnote{Research partly supported by NSA Grant H98230-13-1-0224.}}

\title{Eigenvalues of Non-Regular Linear Quasirandom Hypergraphs}
\author{John Lenz \johnfoot \\ \johnuni \and Dhruv Mubayi \dhruvfoot \\ \dhruvuni}

\begin{document}

\maketitle

\input{nonregular-intro.tex}

\input{algebraicprops.tex}
\input{c4toeig-nonregular.tex}

\bibliographystyle{abbrv}
\bibliography{refs.bib}

\end{document}

%% file: nonregular-intro.tex
\begin{abstract}
  Chung, Graham, and Wilson proved that a graph is quasirandom if and only if there is a large gap
  between its first and second largest eigenvalue.  Recently, the authors extended this
  characterization to $k$-uniform hypergraphs, but only for the so-called coregular $k$-uniform
  hypergraphs.  In this paper, we extend this characterization to all $k$-uniform hypergraphs, not
  just the coregular ones.  Specifically, we prove that if a $k$-uniform hypergraph satisfies the
  correct count of a specially defined four-cycle, then there is a gap between its first and second
  largest eigenvalue.
\end{abstract}

\section{Introduction} 

The study of quasirandom or pseudorandom graphs was initiated by Thomason~\cite{qsi-thomason87,
qsi-thomason87-2} and then refined by Chung, Graham, and Wilson~\cite{qsi-chung89}, resulting in a
list of equivalent (deterministic) properties of graph sequences which are inspired by $G(n,p)$.
Almost immediately after proving their graph theorem, Chung and Graham~\cite{hqsi-chung12,
hqsi-chung90, hqsi-chung90-2, hqsi-chung91, hqsi-chung92} began investigating a $k$-uniform
hypergraph generalization.  Since then, many authors have studied hypergraph
quasirandomness~\cite{hqsi-austin10, hqsi-conlon12, rrl-frankl92, hqsi-gowers06, hqsi-keevash09,
hqsi-kohayakawa10, hqsi-kohayakawa02, hqsi-lenz-quasi12, hqsi-lenz-poset12, hqsi-towsner14}.

One important $k$-uniform hypergraph quasirandom property is \texttt{Disc}, which states that all
sufficiently large vertex sets have roughly the same edge density as the entire hypergraph.
Kohayakawa, Nagle, R\"odl, and Schacht~\cite{hqsi-kohayakawa10} and Conlon, H\`{a}n, Person, and
Schacht~\cite{hqsi-conlon12} studied \texttt{Disc} and found several properties equivalent to it,
but were not able to find a generalization of a graph property called \texttt{Eig}.  In graphs,
\texttt{Eig} states that the first and second largest (in absolute value) eigenvalues of the
adjacency matrix are separated.  The authors~\cite{hqsi-lenz-quasi12} answered this question by
defining a property \texttt{Eig} for $k$-uniform hypergraphs and showed that it is equivalent to
\texttt{Disc}, but only proved this for so-called coregular sequences.  In this paper we prove this
equivalence for all $k$-uniform hypergraph sequences, not just the coregular ones.  Before stating
our result, we need some definitions.

Let $k \geq 2$ be an integer and let $\pi$ be a proper partition of $k$, by which we mean that $\pi$
is an unordered list of at least two positive integers whose sum is $k$.  For the partition $\pi$ of
$k$ given by $k = k_1 + \dots + k_t$, we will abuse notation by saying that $\pi = k_1 + \dots +
k_t$.  A \emph{$k$-uniform hypergraph with loops} $H$ consists of a finite set $V(H)$ and a
collection $E(H)$ of $k$-element multisets of elements from $V(H)$.  Informally, every edge has size
exactly $k$ but a vertex is allowed to be repeated inside of an edge.    If $F$ and $G$ are
$k$-uniform hypergraphs with loops, a \emph{labeled copy of $F$ in $H$} is an edge-preserving
injection $V(F) \rightarrow V(H)$, i.e.\ an injection $\alpha : V(F) \rightarrow V(H)$ such that if
$E$ is an edge of $F$, then $\{ \alpha(x) : x \in E \}$ is an edge of $H$.  The following is our
main theorem.

\begin{thm} \label{thm:countToEig}
  Let $0 < p < 1$ be a fixed constant and let $\mathcal{H} = \{H_n\}_{n\rightarrow \infty}$ be a
  sequence of $k$-uniform hypergraphs with loops such that $|V(H_n)| = n$ and $|E(H_n)| \geq p
  \binom{n}{k}$.  Let $\pi = k_1 + \dots + k_t$ be a proper partition of $k$ and let $\ell \geq 1$.
  Assume that $\mathcal{H}$ satisfies the property
  \begin{itemize}
    \item \propcycle[4\ell]{$\pi$}:  the number of labeled copies of $C_{\pi,4\ell}$ in $H_n$ is at
      most $p^{|E(C_{\pi,4\ell})|} n^{|V(C_{\pi,4\ell})|} + \linebreak[1]
      o(n^{|V(C_{\pi,4\ell})|})$, where $C_{\pi,4\ell}$ is the hypergraph cycle of type $\pi$ and
      length $4\ell$ defined in Section~\ref{sec:cycles-and-traces}.
  \end{itemize}
  Then $\mathcal{H}$ satisfies the property
  \begin{itemize}
    \item \propeig{$\pi$}: $\lambda_{1,\pi}(H_n) = pn^{k/2} + o(n^{k/2})$ and $\lambda_{2,\pi}(H_n)
      = o(n^{k/2})$, where $\lambda_{1,\pi}(H_n)$ and $\lambda_{2,\pi}(H_n)$ are the first and
      second largest eigenvalues of $H_n$ with respect to $\pi$, defined in
      Section~\ref{sec:nonregular-overview}.
  \end{itemize}
\end{thm}

When Theorem~\ref{thm:countToEig} is combined with \cite{hqsi-lenz-quasi12}, we obtain the following
theorem.

\begin{thm}
  Let $0 < p < 1$ be a fixed constant and let $\mathcal{H} = \{H_n\}_{n\rightarrow\infty}$ be a
  sequence of $k$-uniform hypergraphs with loops such that $|V(H_n)| = n$ and $|E(H_n)| \geq p \binom{n}{k} +
  o(n^k)$.  Let $\pi = k_1 + \dots + k_t$ be a proper partition of $k$.
  The following properties are equivalent:

  \begin{itemize}
    \item \propeig{$\pi$}: $\lambda_{1,\pi}(H_n) = pn^{k/2} + o(n^{k/2})$ and $\lambda_{2,\pi}(H_n)
      = o(n^{k/2})$, where $\lambda_{1,\pi}(H_n)$ and $\lambda_{2,\pi}(H_n)$ are the first and
      second largest eigenvalues of $H_n$ with respect to $\pi$, defined in
      Section~\ref{sec:nonregular-overview}.

    \item \propexpand{$\pi$}: For all $S_i \subseteq \binom{V(H_n)}{k_i}$ where $1 \leq i \leq t$,
      \begin{align*} e(S_1,\dots,S_t) =  p \prod_{i=1}^t \left| S_i \right| + o\left(n^{k}\right)
      \end{align*} where $e(S_1,\dots,S_t)$ is the number of tuples $(s_1,\dots,s_t)$ such
      that $s_1 \cup \dots \cup s_t$ is a hyperedge and $s_i \in S_i$.

    \item \propcount{$\pi$-linear}: If $F$ is an $f$-vertex, $m$-edge, $k$-uniform, $\pi$-linear
      hypergraph, then the number of labeled copies of $F$ in $H_n$ is $p^m n^f + o(n^f)$.
      The definition of $\pi$-linear appears in \cite[Section 1]{hqsi-lenz-quasi12}.

    \item \propcycle{$\pi$}: The number of labeled copies of $C_{\pi,4}$ in $H_n$ is at most
      $p^{|E(C_{\pi,4})|} n^{|V(C_{\pi,4})|} + o(n^{|V(C_{\pi,4})|})$, where $C_{\pi,4}$ is the
      hypergraph four cycle of type $\pi$ which is defined in Section~\ref{sec:cycles-and-traces}.

    \item \propcycle[4\ell]{$\pi$}:  the number of labeled copies of $C_{\pi,4\ell}$ in $H_n$ is at
      most $p^{|E(C_{\pi,4\ell})|} n^{|V(C_{\pi,4\ell})|} + \linebreak[1]
      o(n^{|V(C_{\pi,4\ell})|})$, where $C_{\pi,4\ell}$ is the hypergraph cycle of type $\pi$ and
      length $4\ell$ defined in Section~\ref{sec:cycles-and-traces}.
  \end{itemize}
\end{thm}

The remainder of this paper is organized as follows.  Section~\ref{sec:nonregular-overview} contains
the definitions of eigenvalues we will require from~\cite{hqsi-lenz-quasi12} and also a statement of
the main technical contribution of this note.  Section~\ref{sec:algebraicprops} contains the
algebraic properties required for the proof of Theorem~\ref{thm:countToEig}.
Section~\ref{sec:cycles-and-traces} contains the definition of the cycle $C_{\pi,4\ell}$
from~\cite{hqsi-lenz-quasi12} and finally Section~\ref{sec:nonregular-cycle-to-eig} contains the
proof of Theorem~\ref{thm:countToEig}.

\section{Eigenvalues and Linear Maps} 
\label{sec:nonregular-overview}

In this section, we give the definitions of the first and second largest eigenvalues of a
hypergraph.  These definitions are identical to those given in~\cite{hqsi-lenz-quasi12}.  We also
state one of our main results, which extends to $k$-uniform hypergraphs the fact that in a graph
sequence with density $p$ and $\lambda_2(G) = o(\lambda_1(G))$, the distance between the all-ones
vector and the eigenvector corresponding to the largest eigenvalue is $o(1)$.

\begin{definition} \textbf{(Friedman and
  Wigderson~\cite{ee-friedman95,ee-friedman95-2})} Let $H$ be a $k$-uniform hypergraph with loops.
  The \emph{adjacency map of $H$} is the symmetric $k$-linear map $\tau_H : W^k \rightarrow
  \mathbb{R}$ defined as follows, where $W$ is the vector space over $\mathbb{R}$ of dimension
  $|V(H)|$.  First, for all $v_1, \dots, v_k \in V(H)$, let
  \begin{align*}
    \tau_H(e_{v_1}, \dots, e_{v_k}) = 
    \begin{cases}
      1 & \left\{ v_1, \ldots, v_k \right\} \in E(H), \\
      0 & \text{otherwise},
    \end{cases}
  \end{align*}
  where $e_v$ denotes the indicator vector of the vertex $v$, that is the vector which has a one in
  coordinate $v$ and zero in all other coordinates.  We have defined the value of $\tau_H$ when the
  inputs are standard basis vectors of $W$.  Extend $\tau_H$ to all the domain linearly.
\end{definition}

\begin{definition}
  Let $W$ be a finite dimensional vector space over $\mathbb{R}$, let $\sigma : W^k \rightarrow
  \mathbb{R}$ be any $k$-linear function, and let $\vec{\pi}$ be a proper ordered partition of $k$,
  so $\vec{\pi} = (k_1,\dots,k_t)$ for some integers $k_1,\dots,k_t$ with $t \geq 2$.  Now define a
  $t$-linear function $\sigma_{\vec{\pi}} : W^{\otimes k_1} \times \dots \times W^{\otimes k_t}
  \rightarrow \mathbb{R}$ by first defining $\sigma_{\vec{\pi}}$ when the inputs are basis vectors
  of $W^{\otimes k_i}$ and then extending linearly.  For each $i$, $B_i = \{ b_{i,1} \otimes \cdots
  \otimes b_{i,k_i} : b_{i,j} \, \text{is a standard basis vector of } $W$ \}$ is a basis of
  $W^{\otimes k_i}$, so for each $i$, pick $b_{i,1} \otimes \cdots \otimes b_{i,k_i} \in B_i$ and
  define
  \begin{align*}
    \sigma_{\vec{\pi}} \left( b_{1,1} \otimes \dots \otimes b_{1,k_1}, \dots,
                              b_{t,1} \otimes \dots \otimes b_{t,k_t} \right) 
     = \sigma(b_{1,1},\dots,b_{1,k_1},\dots,b_{t,1},\dots,b_{t,k_t}).
  \end{align*}
  Now extend $\sigma_{\vec{\pi}}$ linearly to all of the domain.
  $\sigma_{\vec{\pi}}$ will be $t$-linear since $\sigma$ is $k$-linear.
\end{definition}

\begin{definition}
  Let $W_1,\dots,W_k$ be finite dimensional vector spaces over $\mathbb{R}$, let $\left\lVert \cdot
  \right\rVert$ denote the Euclidean $2$-norm on $W_i$, and
  let $\phi : W_1 \times \dots \times W_k \rightarrow \mathbb{R}$ be a $k$-linear map.  The
  \emph{spectral norm of $\phi$} is
  \begin{align*}
    \left\lVert \phi \right\rVert 
      = \sup_{\substack{x_i \in W_i \\ \left\lVert x_i \right\rVert = 1}} \left|
          \phi(x_1,\dots,x_k) \right|.
  \end{align*}
\end{definition}

\begin{definition}
  Let $H$ be a $k$-uniform hypergraph with loops and let $\tau = \tau_H$ be the ($k$-linear)
  adjacency map of $H$.  Let $\pi$ be any (unordered) partition of $k$ and let $\vec{\pi}$ be any
  ordering of $\pi$.  The \emph{largest and second largest eigenvalues of $H$ with respect to
  $\pi$}, denoted $\lambda_{1,\pi}(H)$ and $\lambda_{2,\pi}(H)$, are defined as
  \begin{align*}
    \lambda_{1,\pi}(H) := \left\lVert \tau_{\vec{\pi}} \right\rVert
    \quad \text{and} \quad
    \lambda_{2,\pi}(H) := \left\lVert \tau_{\vec{\pi}} - \frac{k!|E(H)|}{n^k} J_{\vec{\pi}}
    \right\rVert.
  \end{align*}
\end{definition}

\begin{definition}
  Let $V_1,\dots,V_t$ be finite dimensional vector spaces over $\mathbb{R}$ and let $\phi,\psi : V_1
  \times \dots \times V_t \rightarrow \mathbb{R}$ be $t$-linear maps.  The \emph{product} of $\phi$
  and $\psi$, written $\phi \ast \psi$, is a $(t-1)$-linear map defined as follows.  Let
  $u_1,\dots,u_{t-1}$ be vectors where $u_i \in V_i$. Let $\{b_1,\dots,b_{\dim(V_t)}\}$ be any
  orthonormal basis of $V_t$.
  \begin{gather*}
    \phi \ast \psi : (V_1 \otimes V_1) \times (V_2 \otimes V_2) \times \dots
                     \times (V_{t-1} \otimes V_{t-1}) \rightarrow \mathbb{R} \\
    \phi \ast \psi(u_1 \otimes v_1, \dots, u_{t-1} \otimes v_{t-1}) 
       := \sum_{j=1}^{\dim(V_t)} \phi(u_1,\dots,u_{t-1},b_j) \psi(v_1,\dots,v_{t-1},b_j)
  \end{gather*}
  Extend the map $\phi \ast \psi$ linearly to all of the domain to produce a $(t-1)$-linear map.
\end{definition}

Lemma~\ref{lem:productbasisindependent} shows that the maps are well defined: the map is the same for any
choice of orthonormal basis by the linearity of $\phi$ and $\psi$.

\begin{definition}
  Let $V_1,\dots,V_t$ be finite dimensional vector spaces over $\mathbb{R}$ and let $\phi : V_1
  \times \dots \times V_t \rightarrow \mathbb{R}$ be a $t$-linear map and let $s$ be an integer $0
  \leq s \leq t-1$.  Define
  \begin{gather*}
    \phi^{2^s} : V_1^{\otimes2^{s}} \times \dots \times V_{t-s}^{\otimes2^{s}} \rightarrow
    \mathbb{R}
  \end{gather*}
  where $\phi^{2^0} := \phi$ and $\phi^{2^s} := \phi^{2^{s-1}} \ast \phi^{2^{s-1}}$.
\end{definition}


\begin{definition}
  Let $V_1,\dots,V_t$ be finite dimensional vector spaces over $\mathbb{R}$ and let $\phi : V_1
  \times \dots \times V_t \rightarrow \mathbb{R}$ be a $t$-linear map and define $A[\phi^{2^{t-1}}]$
  to be the following square matrix/bilinear map.  Let $u_1,\dots,u_{2^{t-2}},v_1,\dots,v_{2^{t-2}}$
  be vectors where $u_i, v_i \in V_1$.
  \begin{gather*}
    A[\phi^{2^{t-1}}] : V_1^{\otimes 2^{t-2}} \times V_1^{\otimes 2^{t-2}} \rightarrow \mathbb{R} \\
    A[\phi^{2^{t-1}}](u_1 \otimes \dots \otimes u_{2^{t-2}}, v_1 \otimes \dots v_{2^{t-2}})
    := \phi^{2^{t-1}}(u_1 \otimes v_1 \otimes u_2 \otimes v_2 \otimes \dots \otimes u_{2^{t-2}}
    \otimes v_{2^{t-2}}).
  \end{gather*}
  Extend the map linearly to the entire domain to produce a bilinear map.
\end{definition}

Lemma~\ref{lem:powerismatrix} below proves that $A[\phi^{2^{t-1}}]$ is a square symmetric real
valued matrix.  The following is the main algebraic result required for the proof of
Theorem~\ref{thm:countToEig}.

\begin{prop} \label{prop:algebraicfacts}
  Let $\{\psi_r\}_{r\rightarrow\infty}$ be a sequence of symmetric $k$-linear maps, where $\psi_r :
  V_r^k \rightarrow \mathbb{R}$, $V_r$ is a vector space over $\mathbb{R}$ of finite dimension, and
  $\dim(V_r) \rightarrow \infty$ as $r \rightarrow \infty$. Let $\hat{1}$ denote the all-ones vector
  in $V_r$ scaled to unit length and let $J : V_r^k \rightarrow \mathbb{R}$ be the $k$-linear
  all-ones map.  Let $\pi$ be a proper (unordered) partition of $k$, and assume that for every
  ordering $\vec{\pi}$ of $\pi$,
  \begin{align*}
    \lambda_1(A[\psi_{\vec{\pi}}^{2^{t-1}}]) 
    &= (1+o(1))\psi\left(\hat{1},\dots,\hat{1} \right)^{2^{t-1}}, \\
    \lambda_2(A[\psi_{\vec{\pi}}^{2^{t-1}}]) 
    &= o\left( \lambda_1(A[\psi_{\vec{\pi}}^{2^{t-1}}]) \right).
  \end{align*}
  Then for every ordering $\vec{\pi}$ of $\pi$,
  \begin{align*}
    \left\lVert \psi_{\vec{\pi}} -   qJ_{\vec{\pi}} \right\rVert = o(\psi(\hat{1},\dots,\hat{1})),
  \end{align*}
  where $q = \dim(V_r)^{-k/2} \psi(\hat{1},\dots,\hat{1})$.
\end{prop}

For graphs, $A[\tau^2]$ is the adjacency matrix squared so Proposition~\ref{prop:algebraicfacts}
states that $\lVert A - \linebreak[1] \frac{2|E(G)|}{n^2} J \rVert = o(\sqrt{\lambda_1(A^2)})$, exactly
what is proved by Chung, Graham, and Wilson (see the bottom of page 350 in \cite{qsi-chung89}).  The
proof of Proposition~\ref{prop:algebraicfacts} appears in the next section.

%% file: algebraicprops.tex
\section{Algebraic properties of multilinear maps} 
  \label{sec:algebraicprops}

In this section we prove several algebraic facts about multilinear maps, including
Proposition~\ref{prop:algebraicfacts}.  Throughout this section, $V$ and $V_i$ are finite
dimensional vector spaces over $\mathbb{R}$.  Also in this section we make no distinction between
bilinear maps and matrices, using whichever formulation is convenient. We will use a symbol $\cdot$
to denote the input to a linear map; for example, if $\phi : V_1 \times V_2 \times V_3 \rightarrow
\mathbb{R}$ is a trilinear map and $x_1 \in V_1$ and $x_2 \in V_2$, then by the expression
$\phi(x_1,x_2,\cdot)$ we mean the linear map from $V_3$ to $\mathbb{R}$ which takes a vector $x_3
\in V_3$ to $\phi(x_1,x_2,x_3)$. Lastly, we use several basic facts about tensors, all of which
follow from the fact that for finite dimensional spaces, the tensor product of $V$ and $W$ is the
vector space over $\mathbb{R}$ of dimension $\dim(V) \dim(W)$.  For example, if $x$ and $y$ are unit
length, then $x \otimes y$ is also unit length.

\subsection{Preliminary Lemmas}

\begin{lemma} \label{lem:linearisdotproduct}
  Let $\phi : V \rightarrow \mathbb{R}$ be a linear map.  There exists a vector $v$ such that
  $\phi = \dotp{v}{\cdot}$.
\end{lemma}

\begin{proof} 
$v$ is the vector dual to $\phi$ in the dual of the vector space $V$.  Alternatively,
let the $i$th coordinate of $v$ be $\phi(e_i)$, since then for any $x$,
\begin{align*}
  \phi(x) = \phi \left( \sum \dotp{x}{e_i} e_i \right)
   = \sum \dotp{x}{e_i} \phi(e_i) = \sum \dotp{x}{e_i} \dotp{v}{e_i} = \dotp{x}{v}.
\end{align*}
\end{proof} 

\begin{lemma} \label{lem:productbasisindependent}
  Let $\phi,\psi : V_1 \times \dots \times V_t \rightarrow \mathbb{R}$ be $t$-linear maps.  The maps
  $\phi \ast \psi$ and $A[\phi^{2^{t-1}}]$ are well defined.  Also, $\phi \ast \psi$ is basis
  independent in the sense that the definition of $\phi \ast \psi$ is independent of the choice of
  orthonormal basis $b_1,\dots,b_t$ of $V_t$.
\end{lemma}

\begin{proof} 
First, extending the definitions of $\phi \ast \psi$ and $A[\phi^{2^{t-1}}]$ linearly to the entire
domain (non-simple tensors) is well defined, since $\phi$ and $\psi$ are linear.  That is, write
each $u_i$ and $v_i$ in terms of some orthonormal basis and expand each tensor in $V_i
\otimes V_i$ also in terms of this basis.  The linearity of $\phi$ and $\psi$ then shows that the
definitions of $\phi \ast \psi$ and $A[\phi^{2^{t-1}}]$ are well defined and linear.  To see basis
independence of $\phi \ast \psi$, by Lemma~\ref{lem:linearisdotproduct} the linear map
$\phi(u_1,\dots,u_{t-1},\cdot) : V_t \rightarrow \mathbb{R}$ equals $\dotp{u'}{\cdot}$ for some
vector $u'$.  Similarly, $\psi(v_1,\dots,v_t,\cdot)$ equals $\dotp{v'}{\cdot}$ for some vector $v'$.
Then
\begin{align*}
    (\phi \ast \psi)(u_1\otimes v_1, \dots, u_{t-1} \otimes v_{t-1}) 
    = \sum_{i=1}^{\dim(V_t)} \dotp{u'}{b_i} \dotp{v'}{b_i}
    = \dotp{u'}{v'}.
\end{align*}
The last equality is valid for any orthonormal basis, since the dot product of $u'$ and $v'$ sums
the product of the $i$th coordinate of $u'$ in the basis $\{b_1,\dots,b_{\dim(V_t)}\}$ with the
$i$th coordinate of $v'$ in the basis $\{b_1,\dots,b_{\dim(V_t)}\}$.
\end{proof} 

\begin{definition}
  For $s \geq 0$ and $V$ a finite dimensional vector space over $\mathbb{R}$, define the vector
  space isomorphism $\Gamma_{V,s} : V^{\otimes 2^s} \rightarrow V^{\otimes 2^s}$ as follows.  If $s = 0$,
  define $\Gamma_{V,0}$ to be the identity map.  If $s \geq 1$, let $\{b_1,\dots,b_{\dim(V)}\}$ be any
  orthonormal basis of $V$ and define for all $(i_1, \dots, i_{2^{s-1}}, j_1, \dots, j_{2^{s-1}})
  \in [\dim(V)]^{2^s}$,
  \begin{align} \label{eq:algdefofgamma}
    \Gamma_{V,s}(b_{i_1} \otimes b_{j_1} \otimes \dots \otimes b_{i_{2^{s-1}} } \otimes b_{j_{2^{s-1}} })
    = b_{j_1} \otimes b_{i_1} \otimes \dots \otimes b_{j_{2^{s-1}} } \otimes b_{i_{2^{s-1}} }.
  \end{align}
  Extend $\Gamma_{V,s}$ linearly to all of $V^{\otimes 2^s}$.
\end{definition}

\begin{remarks}
  $\Gamma_{V,s}$ is a vector space isomorphism since it restricts to a bijection of an orthonormal basis
  to itself.  Also, it is easy to see that $\Gamma_{V,s}$ is well defined and independent of the choice of
  orthonormal basis, since each $b_i$ can be written as a linear combination of an orthonormal basis
  $\{b'_1,\dots,b'_{\dim(V)}\}$ and \eqref{eq:algdefofgamma} can be expanded using linearity.  For
  notational convenience, we will usually drop the subscript $V$ and write $\Gamma_s$ for
  $\Gamma_{V,s}$.
\end{remarks}

\begin{lemma} \label{lem:powersymmetric}
  Let $\phi : V_1 \times \dots \times V_t \rightarrow \mathbb{R}$ be a $t$-linear map, let $0 \leq s
  \leq t-1$, and let $x_1 \in V_1^{\otimes 2^s},\dots,x_{t-s} \in V_{t-s}^{\otimes 2^s}$.  Then
  \begin{align*}
    \phi^{2^s}(x_1,\dots,x_{t-s}) = \phi^{2^s}(\Gamma_s(x_1),\dots,\Gamma_s(x_{t-s})).
  \end{align*}
\end{lemma}

\begin{proof} 
By induction on $s$.  The base case is $s = 0$ where $\Gamma_0$ is the identity map. Expand the
definition of $\phi^{2^{s+1}}$ and use induction to obtain
\begin{align*}
  \phi^{2^{s+1}}(x_1 \otimes &y_1,\dots,x_{t-s-1} \otimes y_{t-s-1}) 
  = \sum_{j=1}^{\dim(V_{t-s}^{\otimes 2^s})} 
  \phi^{2^s}(x_1,\dots,x_{t-s-1},b_j) \phi^{2^s}(y_1,\dots,y_{t-s-1},b_j) \\
  &= \sum_{j=1}^{\dim(V_{t-s}^{\otimes 2^s})} 
  \phi^{2^s}\big(\Gamma_s(x_1),\dots,\Gamma_s(x_{t-s-1}),\Gamma_s(b_j)\big)
  \phi^{2^s}\big(\Gamma_s(y_1),\dots,\Gamma_s(y_{t-s-1}),\Gamma_s(b_j)\big).
\end{align*}
But since $\Gamma_s$ is a vector space isomorphism, $\{ \Gamma_s(b_1), \dots,
\Gamma_s(b_{\dim(V_{t-s}^{\otimes 2^{s}})}) \}$ is an orthonormal basis of $V_{t-s}^{\otimes 2^{s}}$.
Thus Lemma~\ref{lem:productbasisindependent} shows that
\begin{align*}
  \sum_{j=1}^{\dim(V_{t-s}^{\otimes 2^s})} 
  &\phi^{2^s}\big(\Gamma_s(x_1),\dots,\Gamma_s(x_{t-s-1}),\Gamma_s(b_j)\big)
  \phi^{2^s}\big(\Gamma_s(y_1),\dots,\Gamma_s(y_{t-s-1}),\Gamma_s(b_j)\big) \\
  &= \phi^{2^{s+1}}\big(\Gamma_s(x_1) \otimes \Gamma_s(y_1), \dots, \Gamma_s(x_{t-s-1}) \otimes
  \Gamma_s(y_{t-s-1})\big)
\end{align*}
Finally, $\Gamma_s(x_i) \otimes \Gamma_s(y_i) = \Gamma_{s+1}(x_i \otimes y_i)$ (write $x_i$ and $y_i$
as linear combinations, expand $\Gamma_{s+1}(x_i \otimes y_i)$ using linearity, and apply
\eqref{eq:algdefofgamma}). Thus $\phi^{2^{s+1}}(x_1 \otimes y_1,\dots,x_{t-s-1} \otimes y_{t-s-1}) =
\phi^{2^{s+1}}(\Gamma_{s+1}(x_1 \otimes y_1), \dots, \Gamma_{s+1}(x_{t-s-1} \otimes y_{t-s-1}))$,
completing the proof.
\end{proof} 

\begin{lemma} \label{lem:powerismatrix}
  Let $V_1,\dots,V_t$ be finite dimensional vector spaces over $\mathbb{R}$.
  If $\phi : V_1 \times \dots \times V_t \rightarrow \mathbb{R}$ is a $t$-linear map, then
  $A[\phi^{2^{t-1}}]$ is a square symmetric real valued matrix.
\end{lemma}

\begin{proof} 
Let $\phi : V_1 \times \dots \times V_t \rightarrow \mathbb{R}$ be a $t$-linear map.
$A[\phi^{2^{t-1}}]$ is a bilinear map from $V_1^{\otimes 2^{t-2}} \times V_1^{\otimes 2^{t-2}}
\rightarrow \mathbb{R}$ and so is a square matrix of dimension $\dim(V_1)^{2^{t-2}}$.
Lemma~\ref{lem:powersymmetric} shows that $A[\phi^{2^{t-1}}]$ is a symmetric matrix, since
\begin{align*}
  A[\phi^{2^{t-1}}](x_1 \otimes \dots \otimes x_{2^{t-2}}, y_1 \otimes \dots \otimes y_{2^{t-2}})
  &= \phi^{2^{t-1}}(x_1 \otimes y_1 \otimes \dots \otimes x_{2^{t-2}} \otimes y_{2^{t-2}}) \\
  &= \phi^{2^{t-1}}(\Gamma(x_1 \otimes y_1 \otimes \dots \otimes x_{2^{t-2}} \otimes y_{2^{t-2}})) \\
  &= \phi^{2^{t-1}}(y_1 \otimes x_1 \otimes \dots \otimes y_{2^{t-2}} \otimes x_{2^{t-2}}) \\
  &= A[\phi^{2^{t-1}}](y_1 \otimes \dots \otimes y_{2^{t-2}},x_1\otimes \dots \otimes x_{2^{t-2}}).
\end{align*}
The above equation is valid for all $x_i,y_i \in V_1$, in particular for all basis elements of $V_1$
which implies that $A[\phi^{2^{t-1}}](w,z) = A[\phi^{2^{t-1}}](z,w)$ for all basis vectors $w,z$ of
$V_1^{\otimes 2^{t-2}}$.  Thus $A[\phi^{2^{t-1}}]$ is a square symmetric real-valued matrix.
\end{proof} 

\begin{lemma} \label{lem:singlepowerupperbound}
  Let $\phi : V_1 \times \dots \times V_t \rightarrow \mathbb{R}$ be a $t$-linear map and let
  $x_1 \in V_1,\dots,x_t \in V_t$ be unit length vectors. Then
  \begin{align*}
    \left| \phi(x_1,\dots,x_t) \right|^2 \leq \left| \phi^2(x_1 \otimes x_1, \dots, x_{t-1}
    \otimes x_{t-1}) \right|.
  \end{align*}
\end{lemma}

\begin{proof} 
Consider the linear map $\phi(x_1, \dots, x_{t-1},\cdot)$ which is a linear map from $V_{t}$ to
$\mathbb{R}$.  By Lemma~\ref{lem:linearisdotproduct}, there exists a vector $w \in V_t$ such that
$\phi(x_1,\dots,x_{t-1},\cdot) = \dotp{w}{\cdot}$.  Now expand out the definition of $\phi^2$:
\begin{align*}
  \phi^2(x_1 \otimes x_1, \dots, x_{t-1} \otimes x_{t-1})
  = \sum_j \left| \phi(x_1, \dots, x_{t-1}, b_j)\right|^2 
  = \sum_j \left| \dotp{w}{b_j} \right|^2 = \dotp{w}{w}
\end{align*}
where the last equality is because $\{b_j\}$ is an orthonormal basis of $V_{t}$. Since
$\left\lVert w \right\rVert = \sqrt{\dotp{w}{w}}$,
\begin{align*}
  \left| \phi^2(x_1 \otimes x_1, \dots, x_{t-1} \otimes x_{t-1})\right| = \left| \dotp{w}{w} \right|
  = \left| \dotp{w}{\frac{w}{\left\lVert w \right\rVert}} \right|^2.
\end{align*}
But since $x_t$ is unit length and $\dotp{w}{\cdot}$ is maximized over the unit ball at vectors
parallel to $w$ (so maximized at $w/\left\lVert w \right\rVert$), $\left|
\dotp{w}{\frac{w}{\left\lVert w \right\rVert}} \right| \geq \left| \dotp{w}{x_t} \right|$. Thus
\begin{align*}
  \left| \phi^2(x_1 \otimes x_1, \dots, x_{t-1} \otimes x_{t-1})\right| = \left|
  \dotp{w}{\frac{w}{\left\lVert w \right\rVert}} \right|^2 \geq \left| \dotp{w}{x_t} \right|^2
  = \left| \phi(x_1,\dots,x_t) \right|^2.
\end{align*}
The last equality used the definition of $w$, that $\phi(x_1,\dots,x_{t-1},\cdot) =
\dotp{w}{\cdot}$.
\end{proof} 

\begin{lemma} \label{lem:powerupperbound}
  Let $\phi : V_1 \times \dots \times V_t \rightarrow \mathbb{R}$ be a $t$-linear map and let
  $x_1 \in V_1,\dots,x_t \in V_t$ be unit length vectors.  Then for $0 \leq s \leq t-1$,
  \begin{align*}
    \left| \phi(x_1,\dots,x_t) \right|^{2^s} \leq \left| \phi^{2^s}(
    \underbrace{x_1 \otimes \dots \otimes x_1}_{2^{s}}, \dots,
    \underbrace{x_{t-s} \otimes \dots \otimes x_{t-s}}_{2^{s}}
    ) \right|
  \end{align*}
  which implies that
  \begin{align*}
    \left| \phi(x_1,\dots,x_t) \right|^{2^{t-1}} \leq \left| A[\phi^{2^{t-1}}](
    \underbrace{x_1 \otimes \dots \otimes x_1}_{2^{t-2}},
    \underbrace{x_1 \otimes \dots \otimes x_1}_{2^{t-2}}
    ) \right|.
  \end{align*}
\end{lemma}

\begin{proof} 
By induction on $s$.  The base case is $s = 0$ where both sides are equal and the induction step
follows from Lemma~\ref{lem:singlepowerupperbound}.  By definition of $A[\phi^{2^{t-1}}]$,
\begin{align*}
  \left| A[\phi^{2^{t-1}}](
  \underbrace{x_1 \otimes \dots \otimes x_1}_{2^{t-2}},
  \underbrace{x_1 \otimes \dots \otimes x_1}_{2^{t-2}}
  ) \right| =
  \left| \phi^{2^{t-1}}(
  \underbrace{x_1 \otimes \dots \otimes x_1}_{2^{t-1}}
  ) \right|,
\end{align*}
completing the proof.
\end{proof} 

\begin{lemma} \label{lem:largesteigpower}
  Let $V_1,\dots,V_t$ be vector spaces over $\mathbb{R}$ and let $\phi : V_1
  \times \dots \times V_t \rightarrow \mathbb{R}$ be a $t$-linear map.  Then $\left\lVert \phi
  \right\rVert^{2^{t-1}} \leq \lambda_1(A[\phi^{2^{t-1}}])$.
\end{lemma}

\begin{proof} 
Pick $x_1,\dots,x_t$ unit length vectors to maximize $\phi$, so $\phi(x_1,\dots,x_t) = \left\lVert
\phi \right\rVert$.  Then Lemma~\ref{lem:powerupperbound} shows that
\begin{align*}
  \left\lVert \phi \right\rVert^{2^{t-1}} = \left| \phi(x_1,\dots,x_t) \right|^{2^{t-1}}
  \leq \left| A[\phi^{2^{t-1}}](
  \underbrace{x_1 \otimes \dots \otimes x_1}_{2^{t-2}},
  \underbrace{x_1 \otimes \dots \otimes x_1}_{2^{t-2}}
  ) \right|
\end{align*}
Since $x_1 \otimes \dots \otimes x_1$ is unit length, the above expression is upper bounded by the
spectral norm of $A[\phi^{2^{t-1}}]$.
\end{proof} 

\begin{lemma} \label{lem:matrixcloseone}
  Let $\{M_r\}_{r\rightarrow\infty}$ be a sequence of square symmetric real-valued matrices with
  dimension going to infinity where $\lambda_2(M_r) = o(\lambda_1(M_r))$.  Let $u_r$ be a unit
  length eigenvector corresponding to the largest eigenvalue in absolute value of $M_r$.  If $\{ x_r
  \}$ is a sequence of unit length vectors such that $\left| x_r^T M_r x_r \right| = (1+o(1))
  \lambda_1(M_r)$, then
  \begin{align*}
    \left\lVert u_r - x_r \right\rVert = o(1).
  \end{align*}
  Consequently, for any unit length sequence $\{y_r\}$ where each $y_r$ is perpendicular to $x_r$,
  \begin{align*}
    \left| y_r^T M_r y_r\right| = o(\lambda_1(M_r)).
  \end{align*}
\end{lemma}

\begin{proof} 
Throughout this proof, the subscript $r$ is dropped; all terms $o(\cdot)$ should be interpreted as
$r \rightarrow \infty$.  This exact statement was proved by Chung, Graham, and
Wilson~\cite{qsi-chung89}, although they don't clearly state it as such.  We give a proof here for
completeness using slightly different language but the same proof idea: if $x$ projected onto
$u^{\perp}$ is too big then the second largest eigenvalue is too big.  Write $x = \alpha v + \beta
u$ where $v$ is a unit length vector perpendicular to $u$ and $\alpha, \beta \in \mathbb{C}$ and
$\alpha^2 + \beta^2 = 1$ (since $u$ is an eigenvector it might have complex entries).  Let
$\phi(x,y) = x^T M y$ be the bilinear map corresponding to $M$.  Since $u^T M v = \lambda_1 u^T v =
\lambda_1 \dotp{u}{v} = 0$, we have $\phi(u,v) = 0$.
This implies that
\begin{align*}
  \phi(x,x) &= \phi(\alpha v + \beta u, \alpha v + \beta u)
  = \alpha^2 \phi(v,v) + \beta^2 \phi(u,u) + 2\alpha\beta \phi(u,v) \\
  &= \alpha^2 \phi(v,v) + \beta^2 \phi(u,u).
\end{align*}
The second largest eigenvalue of $M$ is the largest eigenvalue of $M - \lambda_1(M) u u^T$ which is the
spectral norm of $M - \lambda_1(M) uu^T$. Thus
\begin{align} \label{eq:matrixcloseonesecondbound}
  |\phi(v,v)| = |v^T M v| = |v^T (M - \lambda_1(M) uu^T) v| \leq \lambda_2(M).
\end{align}
Using that $\phi(u,u) = \lambda_1(M)$ and the triangle inequality, we obtain
\begin{align} \label{eq:matrixcloseonealphabeta}
  \left| \phi(x,x) \right| \leq \alpha^2 \lambda_2(M) + \beta^2 \lambda_1(M).
\end{align}
Since $\alpha^2 + \beta^2 = 1$, $|\alpha|$ and $|\beta|$ are between zero and one.  Combining this
with \eqref{eq:matrixcloseonealphabeta} and
$\left| \phi(x,x) \right| =(1+o(1))\lambda_1(M)$ and $\lambda_2(M) = o(\lambda_1(M))$, we must have
$|\beta| = 1 + o(1)$ which in turn implies that $|\alpha| = o(1)$.  Consequently,
\begin{align*}
  \left\lVert u - x \right\rVert^2 = \dotp{u-x}{u-x} = \dotp{u}{u} + \dotp{x}{x} - 2\dotp{u}{x} 
  = 2 - 2\beta = o(1).
\end{align*}

Now consider some $y$ perpendicular to $x$ and similarly to the above, write $y = \gamma w + \delta
u$ for some unit length vector $w$ perpendicular to $u$ and $\gamma,\delta \in \mathbb{C}$ with
$\gamma^2 + \delta^2 = 1$.  Then
\begin{align*}
  \phi(y,y) = \phi(\gamma w + \delta u, \gamma w + \delta u) = \gamma^2 \phi(w,w) + \delta^2 \phi(u,u)
\end{align*}
and as in \eqref{eq:matrixcloseonesecondbound}, we have $\left| \phi(w,w) \right| \leq
\lambda_2(M)$.  Thus
\begin{align*}
  \left| \phi(y,y)\right| \leq \gamma^2  \lambda_2(M) + \delta^2  \lambda_1(M).
\end{align*}
We want to conclude that the above expression is $o(\lambda_1(M))$.
Since $\lambda_2(M) = o(\lambda_1(M))$, we must prove that $\left|\delta\right| = o(1)$ to complete
the proof.
\begin{align*}
  \delta = \dotp{y}{u} = \dotp{y}{\frac{x - \alpha v}{\beta}} = \frac{1}{\beta} \big( \dotp{y}{x}
  - \alpha \dotp{y}{v} \big) = \frac{-\alpha \dotp{y}{v}}{\beta}.
\end{align*}
But $\left|\alpha\right| = o(1)$, $\left| \beta \right| = 1 +
o(1)$, and $\lVert y \rVert = \lVert v \rVert = 1$ so $\left| \delta \right| = o(1)$ as required.
\end{proof} 

\begin{lemma} \label{lem:defofJ}
  Let $J : V_1 \times \dots \times V_t \rightarrow \mathbb{R}$ be the all-ones map and let
  $\vec{1}_i$ be the all-ones vector in $V_i$.  Then for all $x_1,\dots,x_t$ with $x_i \in V_i$,
  \begin{align} \label{eq:Jisdotproduct}
    J(x_1,\dots,x_t) = \dotp{\vec{1}_1}{x_1} \cdots \dotp{\vec{1}_t}{x_t}.
  \end{align}
\end{lemma}

\begin{proof} 
If $x_1,\dots,x_t$ are standard basis vectors, then the left and right hand side of
\eqref{eq:Jisdotproduct} are the same.  By linearity, \eqref{eq:Jisdotproduct} is then the same for
all $x_1,\dots,x_t$.
\end{proof} 

\subsection{Proof of Proposition~\ref{prop:algebraicfacts}} 

\begin{proof}[Proof of Proposition~\ref{prop:algebraicfacts}] 
Again throughout this proof, the subscript $r$ is dropped; all terms $o(\cdot)$ should be
interpreted as $r \rightarrow \infty$.  Let $\hat{1}$ denote the all-ones vector scaled to unit
length in the appropriate vector space.  Pick an ordering ${\vec{\pi}} = (k_1,\dots,k_t)$ of $\pi$.
The definition of spectral norm is independent of the choice of the ordering for the entries of
${\vec{\pi}}$, so $\left\lVert \psi_{\vec{\pi}} - qJ_{\vec{\pi}} \right\rVert$ is the same for all
orderings.  Let $w_1,\dots,w_t$ be unit length vectors where $(\psi_{\vec{\pi}} -
qJ_{\vec{\pi}})(w_1,\dots,w_t) = \left\lVert \psi_{\vec{\pi}} - qJ_{\vec{\pi}} \right\rVert$ and
write $w_i = \alpha_i y_i + \beta_i \hat{1}$ where $y_i$ is a unit length vector perpendicular to
the all-ones vector and $\alpha_i,\beta_i \in \mathbb{R}$ with $\alpha_i^2 + \beta_i^2 = 1$.  Then
\begin{align}
  \left\lVert \psi_{\vec{\pi}} - qJ_{\vec{\pi}} \right\rVert &= (\psi_{\vec{\pi}} - qJ_{\vec{\pi}})(w_1,\dots,w_t) 
  = (\psi_{\vec{\pi}} - qJ_{\vec{\pi}})(\alpha_1 y_1 + \beta_1 \hat{1}, \dots,
  \alpha_t y_t + \beta_t \hat{1})  \nonumber \\
  &= \psi_{\vec{\pi}}(\alpha_1 y_1 + \beta_1 \hat{1}, \dots, \alpha_t y_t + \beta_t \hat{1})
  - q\dim(V_r)^{k/2} \prod \beta_i. \label{eq:algfactsfirstexpansion}
\end{align}
The last equality used that $y_i$ is perpendicular to $\hat{1}$, so Lemma~\ref{lem:defofJ} implies
that if $y_i$ appears as input to $J_{\vec{\pi}}$ then the outcome is zero no matter what the other
vectors are.  Thus the only non-zero term involving $J_{\vec{\pi}}$ is
$J_{\vec{\pi}}(\hat{1},\dots,\hat{1}) = \dim(V_r)^{k/2}$.  Note that $\psi(\hat{1},\dots,\hat{1}) =
\psi_{\vec{\pi}}(\hat{1},\dots,\hat{1})$ since the all-ones vector scaled to unit length in
$V^{\otimes k_i}$ is the tensor product of the all-ones vector scaled to unit length in $V$.
Inserting $q = \dim(V_r)^{-k/2} \psi_{\vec{\pi}}(\hat{1},\dots,\hat{1})$ in
\eqref{eq:algfactsfirstexpansion}, we obtain
\begin{align} \label{eq:boundspectral}
  \left\lVert \psi_{\vec{\pi}} - qJ_{\vec{\pi}} \right\rVert =
    \psi_{\vec{\pi}}(\alpha_1 y_1 + \beta_1 \hat{1}, \dots, \alpha_t y_t + \beta_t \hat{1}) - 
    \left(\prod \beta_i\right)\psi_{\vec{\pi}}(\hat{1},\dots,\hat{1}).
\end{align}
Now consider expanding $\psi_{\vec{\pi}}$ in \eqref{eq:boundspectral} using linearity; the term
$(\prod \beta_i) \psi_{\vec{\pi}}(\hat{1},\dots,\hat{1})$ cancels, so all terms include at least one
$y_i$.  We claim that each of these terms is small; the following claim finishes the proof, since
$\left\lVert \psi_{\vec{\pi}} - qJ_{\vec{\pi}} \right\rVert$ is the sum of terms each of which
$o(\psi(\hat{1},\dots,\hat{1}))$.

\medskip 

\noindent \textit{Claim}: If $z_1,\dots,z_{i-1},z_{i+1},\dots,z_t$ are unit length vectors, then
\begin{align*}
  \left| \psi_{\vec{\pi}}(z_1,\dots,z_{i-1},y_i,z_{i+1},\dots,z_t)\right| =
  o(\psi(\hat{1},\dots,\hat{1})).
\end{align*}

\noindent\textit{Proof.} Change the ordering of ${\vec{\pi}}$ to an ordering ${\vec{\pi}}'$ that differs from
${\vec{\pi}}$ by swapping $1$ and $i$.  Since $\psi$ is symmetric,
\begin{align} \label{eq:algfactreorderpi}
  \psi_{\vec{\pi}}(z_1,\dots,z_{i-1},y_i,z_{i+1},\dots,z_t) =
  \psi_{\vec{\pi}'}(y_i,z_2,\dots,z_{i-1},z_1,z_{i+1},\dots,z_t).
\end{align}
Therefore proving the claim comes down to bounding
$\psi_{\vec{\pi}'}(y_i,z_2,\dots,z_{i-1},z_1,z_{i+1},\dots,z_t)$, which is a combination of
Lemma~\ref{lem:powerupperbound} and Lemma~\ref{lem:matrixcloseone} as follows.  For the remainder of
this proof, denote by $A$ the matrix $A[\psi_{\vec{\pi}'}^{2^{t-1}}]$.  By assumption, we have
$\lambda_2(A) = o(\lambda_1(A))$ so Lemma~\ref{lem:matrixcloseone} can be applied to the matrix
sequence $A$.  Next we would like to show that we can use $\hat{1}$ for $x$ in the statement of
Lemma~\ref{lem:matrixcloseone}; i.e.\ that $A(\hat{1},\hat{1}) = (1+o(1))\lambda_1(A)$.  By
Lemma~\ref{lem:powerupperbound} and the assumption $\lambda_1(A) =
(1+o(1))\psi(\hat{1},\dots,\hat{1})^{2^{t-1}}$, we have
\begin{align*}
  \left| \psi_{\vec{\pi}'}(\hat{1},\dots,\hat{1}) \right|^{2^{t-1}} \leq
  \left| A(\hat{1},\hat{1}) \right| \leq \lambda_1(A)
  = (1+o(1))\psi\left(\hat{1},\dots,\hat{1} \right)^{2^{t-1}}.
\end{align*}
Using the definition of $\psi_{\vec{\pi}'}$, we have $\psi_{\vec{\pi}'}(\hat{1},\dots,\hat{1}) =
\psi(\hat{1},\dots,\hat{1})$, which implies asymtotic equality through the above equation.  In
particular, $|A(\hat{1},\hat{1})| = (1+o(1))\lambda_1(A)$ which is the condition in
Lemma~\ref{lem:matrixcloseone} for $x = \hat{1}$.  Lastly, to apply Lemma~\ref{lem:matrixcloseone}
we need a vector $y$ perpendicular to $\hat{1}$.  The vector $y_i \otimes \dots \otimes y_i \in
V^{\otimes k_i 2^{t-2}}$ is pependicular to $\hat{1}$ (in $V^{\otimes k_i 2^{t-2}}$) since $y_i$
itself is perpendicular to $\hat{1}$ (in $V^{\otimes k_i})$.  Thus Lemma~\ref{lem:matrixcloseone}
implies that
\begin{align} \label{eq:algfactsboundAaty}
  \left| A(\underbrace{y_i \otimes \dots \otimes y_i}_{2^{t-2}}, 
  \underbrace{y_i \otimes \dots \otimes y_i}_{2^{t-2}}
  ) \right| =
  o(\lambda_1(A)).
\end{align}
Using Lemma~\ref{lem:powerupperbound} again shows that
\begin{align*}
  \left| \psi_{\vec{\pi}'}(y_i,z_2,\dots,z_{i-1},z_1,z_{i+1},\dots,z_t) \right|^{2^{t-1}} \leq \left|
  A(\underbrace{y_i \otimes \dots \otimes y_i}_{2^{t-2}}, 
  \underbrace{y_i \otimes \dots \otimes y_i}_{2^{t-2}}
  ) \right|.
\end{align*}
Combining this equation with \eqref{eq:algfactreorderpi} and \eqref{eq:algfactsboundAaty} shows that
$\left| \psi_{\vec{\pi}}(z_1,\dots,z_{i-1},y_i,z_{i+1},\dots,z_t)\right|^{2^{t-1}} =
o(\lambda_1(A))$.  By assumption, $\lambda_1(A) = (1+o(1))\psi(\hat{1},\dots,\hat{1})^{2^{t-1}}$,
completing the proof of the claim.

%
\end{proof} 

%% file: c4toeig-nonregular.tex
\section{Cycles and Traces} 
\label{sec:cycles-and-traces}

A key result we require from~\cite{hqsi-lenz-quasi12} relates the count of the number of cycles of
type $\pi$ and length $4\ell$ to the trace of the matrix $A[\tau_{\vec{\pi}}^{2^{t-1}}]^{2\ell}$.  We
will use this result (Proposition~\ref{prop:powerscountcycles} below) as a black box, but for
completeness we first give the definition of the cycle $C_{\pi,4\ell}$.

\begin{definition}
  Let $\vec{\pi} = (1,\dots,1)$ be the ordered partition of $t$ into $t$ parts. Define the
  \emph{step of type $\vec{\pi}$}, denoted $S_{\vec{\pi}}$, as follows.  Let $A$ be a vertex set of
  size $2^{t-1}$ where elements are labeled by binary strings of length $t-1$ and let
  $B_{2},\dots,B_t$ be disjoint sets of size $2^{t-2}$ where elements are labeled by binary strings
  of length $t-2$.  The vertex set of $S_{\vec{\pi}}$ is the disjoint union $A \dot\cup B_{2}
  \dot\cup \cdots \dot\cup B_t$. Make $\{a,b_{2},\dots,b_t\}$ a hyperedge of $S_{\vec{\pi}}$ if $a
  \in A$, $b_j \in B_j$, and the code for $b_{j+1}$ is equal to the code formed by removing the
  $j$th bit of the code for $a$.

  For a general $\vec{\pi} = (k_1,\dots,k_t)$, start with $S_{(1,\dots,1)}$ and enlarge each vertex
  into the appropriate size; that is, a vertex in $A$ is expanded into $k_1$ vertices and each
  vertex in $B_j$ is expanded into $k_j$ vertices.  More precisely, the vertex set of
  $S_{\vec{\pi}}$ is $(A \times [k_1]) \dot\cup (B_2 \times [k_2]) \dot\cup \cdots \dot\cup (B_t
  \times [k_t])$, and if $\{a,b_2,\dots,b_t\}$ is an edge of $S_{(1,\dots,1)}$, then
  $\{(a,1),\dots,(a,k_1),(b_2,1),\dots, \linebreak[1] (b_2,k_2),\dots, \linebreak[1]
  (b_t,1),\dots,(b_t,k_t)\}$ is a hyperedge of $S_{\vec{\pi}}$.

  This defines the \emph{step of type $\vec{\pi}$}, denoted $S_{\vec{\pi}}$.  Let $A^{(0)}$ be the
  ordered tuple of vertices of $A$ in $S_{\vec{\pi}}$ whose binary code ends with zero and $A^{(1)}$
  the ordered tuple of vertices of $A$ whose binary code ends with one, where vertices are listed in
  lexicographic order within each $A^{(i)}$.  These tuples $A^{(0)}$ and $A^{(1)}$ are the two
  \emph{attach tuples of $S_{\vec{\pi}}$}
\end{definition}

\cite[Figures 1 and 2]{hqsi-lenz-quasi12} contains figures of steps for various $k$ and $\pi$.

\begin{definition}
  Let $\ell \geq 1$.  The \emph{path of type $\vec{\pi}$ of length $2\ell$}, denoted
  $P_{\vec{\pi},2\ell}$, is the hypergraph formed from $\ell$ copies of $S_{\vec{\pi}}$ with
  successive attach tuples identified.  That is, let $T_1,\dots,T_{\ell}$ be copies of
  $S_{\vec{\pi}}$ and let $A^{(0)}_{i}$ and $A^{(1)}_{i}$ be the attach tuples of $T_i$.  The
  hypergraph $P_{\vec{\pi},2\ell}$ is the hypergraph consisting of $T_1,\dots,T_{\ell}$ where the
  vertices of $A^{(1)}_{i}$ are identified with $A^{(0)}_{i+1}$ for every $1 \leq i \leq \ell - 1$.
  (Recall that by definition, $A^{(1)}_i$ and $A^{(0)}_{i+1}$ are tuples (i.e.\ ordered lists) of
  vertices, so the identification of $A^{(1)}_{i}$ and $A^{(0)}_{i+1}$ identifies the corresponding
  vertices in these tuples.) The \emph{attach tuples of $P_{\vec{\pi},2\ell}$} are the tuples
  $A^{(0)}_{1}$ and $A^{(1)}_{\ell}$.
\end{definition}

\begin{definition}
  Let $\ell \geq 2$.  The \emph{cycle of type $\pi$ and length $2\ell$}, denoted $C_{\pi,2\ell}$,
  is the hypergraph formed by picking any ordering $\vec{\pi}$ of $\pi$ and identifying the attach
  tuples of $P_{\vec{\pi},2\ell}$.
\end{definition}

Figure~\ref{fig:cycle} and \cite[Figures 3 and 4]{hqsi-lenz-quasi12} contains figures of paths and
cycles for various $k$ and $\pi$.  The definition of $C_{\pi,2\ell}$ is independent of the ordering
$\vec{\pi}$; a proof appears in~\cite{hqsi-lenz-quasi12-arxiv}.

\begin{figure}[ht] 
\begin{center}
\begin{tikzpicture}[scale=0.7]
  \tikzstyle{foot}=[font=\footnotesize]

  \node (v00) at (0,0) [vertex] {};
  \node (v10) at (3,3) [vertex] {};
  \node (v11) at (6,0) [vertex] {};
  \node (v01) at (3,-3) [vertex] {};
 
  \draw (0,-3) +(45:0.5) node (s00) [vertex] {};
  \draw (0,-3) +(225:0.5) node (s10) [vertex] {};
  \draw (6,3) +(45:0.5) node (s11) [vertex] {};
  \draw (6,3) +(225:0.5) node (s01) [vertex] {};

  \draw (0,3) +(135:0.5) node (d10) [vertex] {};
  \draw (0,3) +(-45:0.5) node (d00) [vertex] {};
  \draw (6,-3) +(135:0.5) node (d01) [vertex] {};
  \draw (6,-3) +(-45:0.5) node (d11) [vertex] {};

  \draw (d10).. controls (v00).. (s10);
  \draw (s00).. controls (v00).. (d00);
  \draw (s10).. controls (v01).. (d11);
  \draw (s00).. controls (v01).. (d01);
  \draw (d01).. controls (v11).. (s01);
  \draw (d11).. controls (v11).. (s11);
  \draw (d00).. controls (v10).. (s01);
  \draw (d10).. controls (v10).. (s11);
\end{tikzpicture}
\caption{$C_{(1,1,1),4}$}
\label{fig:cycle}
\end{center}
\end{figure}
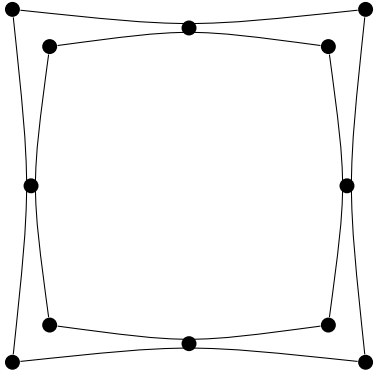 

\begin{definition}
  Let $\ell \geq 2$.  A \emph{circuit of type $\pi$ of length $2\ell$ in a hypergraph $H$} is a
  function $f : V(C_{\pi,2\ell}) \rightarrow V(H)$ that preserves edges.  Informally, a circuit is a
  cycle where the vertices are not necessarily distinct.
\end{definition}

\begin{prop} \label{prop:powerscountcycles} \powerscountcyclescite
  Let $H$ be a $k$-uniform hypergraph, let $\vec{\pi}$ be a proper ordered partition of $k$, and let
  $\ell \geq 2$ be an integer.  Let $\tau$ be the adjacency map of $H$. Then
  $\tr{A[\tau_{\vec{\pi}}^{2^{t-1}}]^\ell}$ is the number of labeled circuits of type $\vec{\pi}$
  and length $2\ell$ in $H$.
\end{prop}

\section{\propcycle[4\ell]{$\pi$} $\Rightarrow$ \propeig{$\pi$}} 
\label{sec:nonregular-cycle-to-eig}

In this section, we prove that \propcycle[4\ell]{$\pi$} $\Rightarrow$ \propeig{$\pi$} using
Propositions~\ref{prop:algebraicfacts} and~\ref{prop:powerscountcycles}.

\begin{proof}[\prooftext that {\texttt{Cycle}$_{4\ell}$[$\pi$]} $\Rightarrow$ \propeig{$\pi$}] 
\label{proof:c4toeig}
Let $\mathcal{H} = \{H_n\}_{n \rightarrow \infty}$ be a sequence of hypergraphs and let $\tau_n$ be
the adjacency map of $H_n$.  For notational convenience, the subscript on $n$ is dropped below.
Throughout this proof, we use $\hat{1}$ to denote the all-ones vector scaled to unit length.
Wherever we use the notation $\hat{1}$, it is the input to a multilinear map and so $\hat{1}$
denotes the all-ones vector in the appropriate vector space corresponding to whatever space the map
is expecting as input.  This means that in the equations below $\hat{1}$ can stand for different vectors
in the same expression, but attempting to subscript $\hat{1}$ with the vector space (for example
$\hat{1}_{V_3}$) would be notationally awkward.

The proof that \propcycle[4\ell]{$\pi$} $\Rightarrow$ \propeig{$\pi$} comes down to checking the
conditions of Proposition~\ref{prop:algebraicfacts}.  Let $\vec{\pi}$ be any ordering of the entries
of $\pi$.  We will show that the first and second largest eigenvalues of $A =
A[\tau_{\vec{\pi}}^{2^{t-1}}]$ are separated.  Let $m = |E(C_{\pi,4\ell})| = 2\ell 2^{t-1}$ and note
that $|V(C_{\pi,4\ell})| = mk/2$ since $C_{\pi,4\ell}$ is two-regular.  $A$ is a square symmetric
real valued matrix, so let $\mu_1,\dots, \mu_d$ be the eigenvalues of $A$ arranged so that $|\mu_1|
\geq \dots \geq |\mu_d|$, where $d = \dim(A)$.  The eigenvalues of $A^{2\ell}$ are $\mu_1^{2\ell},
\dots, \mu_d^{2\ell}$ and the trace of $A^{2\ell}$ is $\sum_i \mu_i^{2\ell}$.  Since all
$\mu_i^{2\ell} \geq 0$, Proposition~\ref{prop:powerscountcycles} and \propcycle[4\ell]{$\pi$}
implies that
\begin{align} \label{eq:c4toeigboundlambda}
  \mu_1^{2\ell} + \mu_2^{2\ell} 
  &\leq \tr{A^{2\ell}} 
  = \#\{\text{possibly degenerate } C_{\pi,4\ell} \,\, \text{in } H_n\} 
  \leq p^m n^{mk/2} + o(n^{mk/2}).
\end{align}
We now verify the conditions on $\mu_1$ and $\mu_2$ in Proposition~\ref{prop:algebraicfacts}, and to
do that we need to compute $\tau(\hat{1},\dots,\hat{1})$.  Simple computations show that
\begin{align} \label{eq:c4toeigexacttauallones}
  \tau(\hat{1},\dots,\hat{1}) = \tau_{\vec{\pi}}(\hat{1},\dots,\hat{1}) = \frac{k! E(H)}{n^{k/2}}.
\end{align}
Using that $|E(H_n)| \geq p\binom{n}{k} + o(n^k)$, Lemma~\ref{lem:largesteigpower}, and
$\mu_1^{2\ell} \leq p^m n^{mk/2} + o(n^{mk/2})$ from \eqref{eq:c4toeigboundlambda},
\begin{align} \label{eq:c4toeigboundtaunorm}
  pn^{k/2} + o(n^{k/2}) &\leq \frac{k! E(H)}{n^{k/2}} = \tau_{\vec{\pi}}(\hat{1},\dots,\hat{1}) \leq
  \left\lVert \tau_{\vec{\pi}} \right\rVert 
  \leq \mu_1^{1/2^{t-1}} \leq pn^{k/2} + o(n^{k/2}).
\end{align}
This implies equality up to $o(n^{k/2})$ throughout the above expression, so
$\tau(\hat{1},\dots,\hat{1}) = pn^{k/2} + o(n^{k/2})$, $\lambda_{1,\pi}(H_n) = \left\lVert
\tau_{\vec{\pi}} \right\rVert = pn^{k/2} + o(n^{k/2})$, and $\mu_1 = p^{2^{t-1}} n^{k2^{t-2}} +
o(n^{k2^{t-2}})$, so $\mu_1 = (1+o(1))\tau(\hat{1},\dots,\hat{1})^{2^{t-1}}$.

Insert $\mu_1 = p^{2^{t-1}}
n^{k2^{t-2}} + o(n^{k2^{t-2}})$ into \eqref{eq:c4toeigboundlambda} to show that $\mu_2 =
o(n^{k2^{t-2}})$.  Therefore, the conditions of Proposition~\ref{prop:algebraicfacts} are satisfied,
so
\begin{align*}
  \left\lVert \tau_{\vec{\pi}} - qJ_{\vec{\pi}} \right\rVert = o(\tau(\hat{1},\dots,\hat{1})) =
  o(n^{k/2}),
\end{align*}
where $q = n^{-k/2}\tau(\hat{1},\dots,\hat{1})$. Using \eqref{eq:c4toeigexacttauallones}, $q =
k!|E(H)|/n^k$.  Thus $\left\lVert \tau_{\vec{\pi}} - qJ_{\vec{\pi}} \right\rVert =
\lambda_{2,{\pi}}(H_n)$ and the proof is complete.
\end{proof} 

The above proof can be extended to even length cycles in the case when $\vec{\pi} = (k_1,k_2)$ is a
partition into two parts.  For these $\vec{\pi}$, the matrix $A[\tau_{\vec{\pi}}^2]$ can be
shown to be positive semidefinite since $A[\tau_{\vec{\pi}}^2]$ will equal $MM^T$ where $M$ is the
matrix associated to the bilinear map $\tau_{\vec{\pi}}$.  Since $A[\tau_{\vec{\pi}}^2]$ is positive
semidefinite, each $\mu_i \geq 0$ so any power of $\mu_i$ is non-negative.  For partitions into more
than two parts, we don't know if the matrix $A[\tau_{\vec{\pi}}^{2^{t-1}}]$ is always positive
semidefinite or not.